\documentclass[a4paper,11pt]{article}

\usepackage[T1]{fontenc}
\usepackage{lmodern}

\usepackage{dsfont}

\usepackage[latin1]{inputenc}
\usepackage{amsmath}
\usepackage{amsthm}
\usepackage{amssymb}
\usepackage{mathrsfs}
\usepackage{graphicx}
\usepackage[all]{xy}
\usepackage{hyperref}

\usepackage{makeidx}

\usepackage{stmaryrd}
\usepackage{caption}

\usepackage{abstract} 

\newtheorem{thm}{Theorem}[section]

\newtheorem{claim}[thm]{Claim}

\newtheorem{lemma}[thm]{Lemma}

\theoremstyle{definition}

\def\rquotient#1#2{%
	\makeatletter
	\raise.3ex\hbox{$#1$}/\lower.3ex\hbox{$#2$}%
	\makeatother
}	

\makeatletter
\newcommand{\subjclass}[2][2010]{%
	\let\@oldtitle\@title%
	\gdef\@title{\@oldtitle\footnotetext{#1 \emph{Mathematics subject classification.} #2}}%
}
\newcommand{\keywords}[1]{%
	\let\@@oldtitle\@title%
	\gdef\@title{\@@oldtitle\footnotetext{\emph{Key words and phrases.} #1.}}%
}
\makeatother

\newcommand{\Address}{{
		\bigskip
		\small
		
		\textsc{D\'epartement de Math\'ematiques B\^atiment 307, Facult\'e des Sciences d'Orsay, Universit\'e Paris-Sud, F-91405 Orsay Cedex, France.}\par\nopagebreak
		\textit{E-mail address}: \texttt{anthony.genevois@math.u-psud.fr}
		
}}

\makeindex

\title{CAT(0) cube complexes with flat hyperplanes}
\date{\today}
\author{Anthony Genevois}


\begin{document}

\maketitle

\begin{abstract}
In this short note, we show that a group acting geometrically on a CAT(0) cube complex with virtually abelian hyperplane-stabilisers must decompose virtually as a free product of free abelian groups and surface groups. 
\end{abstract}


\section{Introduction}

\noindent
CAT(0) cube complexes first appeared in the monograph \cite{MR919829} as a convenience source of CAT(0) and CAT(-1) spaces, leading to the construction of nonpositively curved spaces on which classical families of groups act (e.g. \cite{MR1303028, MR1389635}) but also to the construction of exotic groups \cite{MR2694733, BurgerMozes}. However, the strength of CAT(0) cube complexes really arose after the recognition of the fundamental role played by \emph{hyperplanes}. This role is twofold. Firstly, it turns out that the geometry of CAT(0) cube complexes essentially reduces to the combinatorics of their hyperplanes \cite{MR1347406}, a point of view which provides powerful tools to attack various questions including Hilbertian geometry \cite{Haagerup, MR1459140, MR2132393}, finiteness properties \cite{MR1465330}, biautomaticity \cite{MR1604899}, Tits' alternative \cite{MR2164832, MR2827012}, separability properties \cite{MR2377497}, the flat closing conjecture \cite{PeriodicFlatsCC, NTY, SpecialHyp, RHspecial}, the rank one rigidity conjecture \cite{MR2827012}. Secondly, CAT(0) cube complexes can be reconstructed from their hyperplanes, leading to easy constructions of CAT(0) cube complexes from \emph{cubulations} of \emph{pocsets} and \emph{spaces with walls} \cite{MR1347406, Roller, MR1668359, MR2197811, MR2059193}. Such constructions allow us to prove that many groups naturally act on CAT(0) cube complexes, including many Artin groups \cite{MR1303028, MR2995171, MR3993762}, graph braid groups \cite{GraphBraidGroups}, Coxeter groups \cite{MR1983376}, small cancellation groups \cite{MR2053602, MR3346927, MR3711135}, Thompson's groups \cite{MR1978047, MR2136028}, random groups \cite{MR2806688, MR3846337}, many 3-manifold groups \cite{MR2931226, MR3217626, MR3358051, MR3758147, MR3874647}, one-relator groups with torsion \cite{MR3118410}, many free-by-cyclic groups \cite{MR3320891, MR3513573}, some Burnside groups \cite{MR3786300}, Cremona groups \cite{CremonaCC}. As a consequence, looking for an action on a CAT(0) cube complex is a useful geometric strategy in order to study a given group, but it also has applications in other areas of mathematics, most famously in low-dimensional topology \cite{MR3104553}. More recently, coarse geometries inspired by CAT(0) cube complexes and their connections with mapping class groups of closed surfaces received a lot of attention (see \cite{Coarsemedian, MR3650081, MR3956144} and their subsequent developments). 

\medskip \noindent
In this short note, we investigate the following natural question: what can be said about a group acting geometrically on a CAT(0) cube complex from the structure of its hyperplane-stabilisers? Notice that, if the cube complexes under consideration are two-dimensional, then hyperplane-stabilisers must be virtually free. Such examples include groups with quite different behaviors, for instance some right-angled Artin or Coxeter groups, some hyperbolic small cancellation groups and some simple groups. Therefore, it seems more reasonable to consider the case where hyperplane-stabilisers are \emph{small} (i.e. they do not contain non-abelian free subgroups), or equivalently, as a consequence of the Tits alternative proved in \cite{MR2164832}, where hyperplane-stabilisers are virtually abelian. 

\medskip \noindent
This question first appears in \cite[Conjecture 14.11]{WisePreprint}, where the author conjectures that, if $X$ is a nonpositively curved cube complex of finite dimension all of whose hyperplanes have virtually abelian fundamental groups, then $\pi_1(X)$ must be virtually abelian or (virtually abelian)-by-(non-elementary quasifuchsian). Our main theorem answers Wise's question under the stronger assumption that the action is cocompact but allowing torsion in the group.

\begin{thm}\label{thm:Main}
Let $G$ be a group acting geometrically on a CAT(0) cube complex. If hyperplane-stabilisers are all virtually abelian, then $G$ virtually decomposes as a free product of free abelian groups and surface groups.
\end{thm}

\noindent
Here, surface groups refer to fundamental groups of orientable surfaces (possibly with boundary), so they correspond to free groups and fundamental groups of closed surfaces. 

\medskip \noindent
Theorem \ref{thm:Main} is proved as follows. Given a group $G$ acting geometrically on a CAT(0) cube complex $X$ with virtually abelian hyperplane-stabilisers, the starting point is to simplify the cubulation following \cite{MR4033506}. As a consequence, we can suppose that the hyperplanes in $X$ decompose as products of unbounded quasi-lines. The number of such factors is referred to as the \emph{rank} of the hyperplane. An elementary observation is that two transverse hyperplanes must have the same rank. As a consequence, if we assume that $G$ is one-ended, all the hyperplanes in $X$ have the same rank $r \geq 0$.
\begin{itemize}
	\item If $r=0$, then $X$ is a tree and $G$ is virtually free. 
	\item If $r \geq 2$, then we show that any two non-transverse hyperplanes are transverse to a common hyperplane. It follows from the rank one rigidity \cite{MR2827012} that $X$ decomposes as a product $X_1 \times \cdots \times X_n$, $n\geq 2$. Since $X$ contains hyperplanes isomorphic to $X_1 \times \cdots \times X_{n-1}$ and $X_2 \times \cdots \times X_n$, it follows that $X_1, \ldots, X_n$ are quasi-lines and that $n=r+1$. The conclusion is that $G$ contains a finite-index subgroup isomorphic to $\mathbb{Z}^{r+1}$.
	\item Finally, if $r=1$, we distinguish two cases. Either $X$ is hyperbolic, and we show that $X$ is quasi-isometric to the hyperbolic plane and that $G$ is virtually a surface group. Or $X$ is not hyperbolic, and we show that $X$ is quasi-isometric to a the Euclidean plane and that $G$ contains a finite-index subgroup isomorphic to $\mathbb{Z}^2$.
\end{itemize}
The conclusion in the latter case is not surprising if we assume that the hyperplanes are not only quasi-lines but geodesic lines. Then it is easy to verify that, in the link of a vertex, no vertex can have degree $\geq 3$ (otherwise we find a hyperplane containing a branching point) and that each vertex has degree $\geq 2$ (otherwise we find a hyperplane with a leaf). In other words, the link of each vertex must be a cycle, which implies that $X$ is a square tessellation of the plane. Then two cases happen: either $X$ is isomorphic to $\mathbb{R}^2$ (endowed with its canonical cubulation) or $X$ contains a vertex of degree $\geq 5$ and it is quasi-isometric to the hyperbolic plane. In full generality, when the hyperplanes of $X$ are only quasi-lines, it follows from the fact that $X$ is one-ended that $\mathbb{R}^2$ or $\mathbb{H}^2$ quasi-isometrically embed into $X$, and we show that the image of such an embedding is necessarily quasi-dense.

\paragraph{Acknowledgements.} This work was supported by a public grant as part of the Fondation Math\'ematique Jacques Hadamard.

\section{Proof of the theorem}

\noindent
We assume that the reader is familiar with CAT(0) cube complexes. We refer to \cite{MR2986461, SageevCAT(0)} for more information. We emphasize that, in the following, hyperplanes are thought of as CAT(0) cube complexes on their own and sometimes as subcomplexes. Formally, this can be justified by noticing that a hyperplane becomes a convex subcomplex in the barycentric subdivision of the cube complex under consideration.

\medskip \noindent
Before turning to the proof of Theorem \ref{thm:Main}, we record the following observation:

\begin{lemma}\label{lem:Crossing}
Let $X$ be a CAT(0) cube complex. There exist convex subcomplexes $\{X_i \mid i \in I\}$ such that:
\begin{itemize}
	\item for all distinct $i,j \in I$, $X_i \cap X_j$ is either empty or a cut vertex;
	\item for every $i \in I$, the crossing graph of $X_i$ is connected and the hyperplanes crossing $X_i$ span a connected component of the crossing graph of $X$.
\end{itemize}
\end{lemma}

\noindent
Recall that the \emph{crossing graph} of a CAT(0) cube complex is the graph whose vertices are the hyperplanes of the cube complex under consideration and whose edges link two hyperplanes whenever they are transverse.

\begin{proof}[Proof of Lemma \ref{lem:Crossing}.]
Let $\{ x_j \mid j \in J\}$ denote the collection of all the cut vertices of $X$ and let $\{X_i \mid i \in I\}$ denote the closures of the connected components of $X \backslash \{x_j \mid j \in J\}$. Then, for all distinct $i,j \in I$, $X_i \cap X_j$ is either empty or a cut vertex; and, for every $i \in I$, $X_i$ is not disconnected by one of its vertices. The latter assertion implies, according to \cite[Lemma 2]{SplittingObstruction}, that the crossing graph of $X_i$ is connected. So the hyperplanes crossing $X_i$ lie in a single connected component of the crossing graph of $X$. Conversely, it is clear that two hyperplanes separated by a cut vertex lie in distinct connected components of the crossing graph of $X$. We conclude that the hyperplanes crossing $X_i$ span a connected component of the crossing graph of $X$.
\end{proof}

\begin{proof}[Proof of Theorem~\ref{thm:Main}.]
Assume that $G$ acts geometrically on a CAT(0) cube complex $X$ such that each hyperplane has a virtually abelian stabiliser. According to \cite[Proposition 3.5]{MR2827012} and \cite[Theorem A]{MR4033506}, we can suppose without loss of generality that the action is \emph{essential} (i.e. the orbit of a point never stays in a neighborhood of a halfspace) and \emph{hyperplane-essential} (i.e. the action on a hyperplane by its stabiliser is always essential). 

\begin{claim}\label{claim:HypProduct}
Each hyperplane is a product of unbounded quasi-lines.
\end{claim}

\noindent
For every hyperplane $J$, $\mathrm{stab}(J)$ acts geometrically on $J$. Because $\mathrm{stab}(J)$ is virtually abelian, it follows from the cubical flat torus theorem proved in \cite[Theorem 3.6]{MR3625111} that $J$ contains a convex subcomplex which decomposes as a product of quasi-lines and on which $\mathrm{stab}(J)$ acts cocompactly. Because $\mathrm{stab}(J)$ acts on $J$ essentially, it follows that $J$ coincides with this subcomplex and that our quasi-lines are all unbounded.

\medskip \noindent
Claim~\ref{claim:HypProduct} allows us to define the \emph{rank} of a hyperplane as the number of factors in its decomposition as a product of unbounded quasi-lines.

\begin{claim}\label{claim:TransverseHyp}
Two transverse hyperplanes have the same rank.
\end{claim}

\noindent
Let $A,B$ be two transverse hyperplanes. According to Claim~\ref{claim:HypProduct}, we can write $A=A_1 \times \cdots \times A_n$ and $B= B_1 \times \cdots \times B_m$ for some convex and unbounded quasi-lines $A_1, \ldots, A_n,B_1, \ldots, B_m \subset X$. Up to reindexing our quasi-lines, we suppose that $A$ crosses $B_1$ and that $B$ crosses $A_1$. Then there exist two points $p \in A_1$ and $q \in B_1$ such that
$$\{p\} \times A_2 \times \cdots \times A_n = A \cap B = \{q\} \times B_2 \times \cdots \times B_m.$$
This equality implies that $n=m$, concluding the proof of our claim.

\medskip \noindent
Claim~\ref{claim:TransverseHyp} shows that all the hyperplanes of a connected component of the crossing graph of $X$ have the same rank. It follows from Lemma \ref{lem:Crossing} that $X$ decomposes as a union of convex subcomplexes $\{X_i \mid i \in I\}$ such that $X_i \cap X_j$ is either empty or a cut vertex for all distinct $i,j \in I$; and such that, for all $i \in I$, $X_i$ has a connected crossing graph and all its hyperplanes have the same rank. Consequently, $G$ acts on the tree whose vertex-set is $I$ and whose edges link $i,j \in I$ if $X_i \cap X_j$ is a cut vertex of $X$. Notice that edge-stabilisers are finite and that vertex-stabilisers are $\{ \mathrm{stab}(X_i) \mid i \in I\}$. Therefore, the proof of Theorem~\ref{thm:Main} reduces to the case where $X$ has a connected crossing graph and all its hyperplanes have the same rank $r \geq 0$. We distinguish several cases depending on the value of $r$.

\medskip \noindent
\textbf{Case 1:} $r=0$. In other words, all the hyperplanes of $X$ are points, i.e. $X$ is a tree. Because the crossing graph of $X$ is connected by assumption, it must be reduced to a single vertex and $G$ must be finite.

\medskip \noindent
\textbf{Case 2:} $r \geq 2$. Here, we want to prove that $G$ is virtually free abelian. We begin by proving the following observation:

\begin{claim}\label{claim:Parallel}
For any two non-transverse hyperplanes $A$ and $B$, there exists a third hyperplane $J$ which is transverse to both $A$ and $B$.
\end{claim}

\noindent
Because the crossing graph of $X$ is connected, we can fix a geodesic $J_0, \ldots, J_n$ from $A$ to $B$ in the crossing graph. For every $0 \leq i \leq n$, we decompose $J_i$ as a product of quasi-lines $J_i^1 \times \cdots \times J_i^r$ as given by Claim~\ref{claim:HypProduct}. Up to reindexing our quasi-lines, we suppose that $J_i$ crosses $J_{i-1}^1$ for every $1 \leq i \leq n$. Given an index $1 \leq i \leq n-1$, the hyperplanes $J_{i-1}$ and $J_{i+1}$ cannot be transverse, since otherwise $J_0, \ldots, J_n$ would not be a geodesic; so $J_{i-1}$ and $J_{i+1}$ have to cross the same factor of $J_i$, namely $J_i^1$. Therefore, for every $0 \leq i \leq n-1$, there exist two points $d_i \in J_i^1$ and $g_{i+1} \in J_{i+1}^1$ such that
$$\{d_i\} \times J_i^2 \times \cdots \times J_i^r = J_i \cap J_{i+1} = \{g_{i+1}\} \times J_{i+1}^2 \times \cdots \times J_{i+1}^r.$$
Because $r \geq 2$, there exists a hyperplane $J$ crossing $J_0^2$. Then $J$ crosses
$$\{d_0\} \times J_0^2 \times \cdots \times J_0^r = J_0 \cap J_1 = \{g_1\} \times J_1^2 \times \cdots \times J_1^r,$$
and we deduce that it also crosses
$$\{d_1 \} \times J_1^2 \times \cdots \times J_1^r = J_1 \cap J_2 = \{g_2\} \times J_2^2 \times \cdots \times J_2^r.$$
By iterating the argument, we conclude that $J$ crosses $J_i$ for every $0 \leq i \leq n$. In particular, $J$ crosses $J_0=A$ and $J_n=B$, concluding the proof of our claim.

\medskip \noindent
The combination of Claim~\ref{claim:Parallel} and \cite[Theorem 6.3]{MR2827012} shows that $X$ must decompose as a product of irreducible subcomplexes $X_1 \times \cdots \times X_n$, $n \geq 2$. The desired conclusion follows from the following observation, which we record for future use:

\begin{claim}\label{claim:NotIrreducible}
If $X$ is not irreducible, then $G$ is virtually free abelian.
\end{claim}

\noindent
Notice that $X$ contains hyperplanes isomorphic to $X_2 \times \cdots \times X_n$ and $X_1 \times \cdots \times X_{n-1}$, so \cite[Proposition~2.6]{MR2827012} implies that $X_1, \ldots, X_n$ are quasi-lines. Therefore, our group $G$ must be quasi-isometric to $\mathbb{Z}^n$, and we conclude from \cite{MR623534, MR741395} that $G$ is virtually free abelian, as desired.

\medskip \noindent
\textbf{Case 3.1:} $r=1$ and $X$ is hyperbolic. If $G$ is virtually free there is nothing to prove, so from now on we assume that $G$ is not virtually free. As a consequence of \cite{MR2146190}, $X$ contains a quasiconvex subspace $Q$ quasi-isometric a hyperbolic plane. Observe that:

\begin{claim}\label{claim:observation}
Let $J$ be a hyperplane. Assume that $\partial Q$ lies in the (Gromov) boundary of a halfspace $J^+$ delimited by $J$. For any two distinct points at infinity $\alpha, \beta \in \partial Q \backslash \{\partial J\}$, every bi-infinite geodesic $\gamma$ between $\alpha$ and $\beta$ lies in $J^+$.
\end{claim}

\noindent
Assume that our geodesic $\gamma$ does not lie in $J^+$. Then $\gamma$ contains an infinite ray $\rho_1$ which is disjoint from $J^+$. Consequently, $J$ separates $\rho_1$ from a geodesic ray $\rho_2$ converging to $\rho_1(+\infty) \in \{\alpha,\beta\} \subset \partial J^+$. Because the Hausdorff distance between $\rho_1$ and $\rho_2$ must be finite, $J$ must contain a ray $\rho_3$ converging $\rho_1(+ \infty)$ as well, so $\alpha$ or $\beta$ necessarily belongs to $\partial J$. This concludes the proof of our claim.

\medskip \noindent
Let $Y$ denote the intersection of all the halfspaces containing $\partial Q$ in their boundaries. As a consequence of Claim \ref{claim:observation}, any geodesic between any two distinct points in 
$$S: = \partial Q \backslash \{ \partial J, \text{ $J$ hyperplane}\}$$
lies in $Y$, so $S \subset \partial Y$. Because $X$ contains only countably many hyperplanes and that the boundary of each hyperplane is finite, we know that $S$ has countable cocardinality in $\partial Q$, which is a circle; in particular, $S$ is dense in $\partial Q$. Since $\partial Y$ must be closed in $\partial X$, it follows that $\partial Y$ contains $\partial Q$. Conversely, $Y$ lies in the convex hull of $Q$ (which coincides with the intersection of all the halfspaces containing $Q$), and we know from \cite[Theorem H]{MR2413337} that the Hausdorff distance between $Q$ and its convex hull is finite because $Q$ is quasiconvex, so $\partial Y$ must lie in $\partial Q$. Hence $\partial Y = \partial Q$.

\medskip \noindent
Thus, we have constructed a convex subcomplex $Y$ quasi-isometric to the hyperbolic plane such that every hyperplane of $X$ either is disjoint from $Y$ or separates $\partial Y$. If $Y$ is a proper subcomplex, then there exists a hyperplane in $X$ which is disjoint from $Y$. Because the crossing graph of $X$ is connected, this implies that there exists a hyperplane $A$ disjoint from $Y$ which is transverse to a hyperplane $B$ which crosses $Y$. On the one hand, $B$ separates $\partial Y$ so it contains a quasi-line $\gamma \subset Y$. And on the other hand, $A$ must be disjoint from $\gamma$, so, because $B$ is essential, there must exist points in $B$ which are arbitrarily far away from $\gamma$. This contradicts the fact that $B$ is a quasi-line. Thus, we have proved that $X=Y$ is quasi-isometric to the hyperbolic plane, and we conclude that $G$ contains a finite-index subgroup isomorphic to the fundamental group of a closed surface of genus $\geq 2$ \cite{MR961162, MR1313451, MR1296353, MR1189862}.

\medskip \noindent
\textbf{Case 3.2:} $r=1$ and $X$ is not hyperbolic. According to \cite[Theorem 3.1]{MR4057355}, there exists a combinatorial isometric embedding $\mathbb{R}^2 \to X$ where $\mathbb{R}^2$ is endowed with its canonical square tessellation. Let $F$ denote the image of such an embedding. If the convex hull of $F$ is a proper subcomplex in $X$, then there must exist a hyperplane of $X$ which does not cross $F$. Because the crossing graph of $X$ is connected, this implies that there exists a hyperplane $A$ disjoint from $F$ which is transverse to a hyperplane $B$ which crosses $F$. On the one hand, $B \cap F$ is a quasi-line; and on the other hand, $A$ must be disjoint from $\gamma$, so, because $B$ is essential, there must exist points in $B$ which are arbitrarily far away from $\gamma$. This contradicts the fact that $B$ is a quasi-line. Thus, we have proved that $X$ is the convex hull of $F$. As a consequence, the crossing graph of $X$ is obtained from the crossing graph of $\mathbb{R}^2$ (which is an infinite bipartite complete graph) by adding new edges. It follows from \cite[Lemma 2.5]{MR2827012} that $X$ is not irreducible, and we conclude from Claim \ref{claim:NotIrreducible} that $G$ is virtually free abelian.

\medskip \noindent
By putting together all the cases above, we conclude that our group $G$ decomposes as a graph of groups such that the edge-groups are finite and the vertex-groups are either virtually free abelian or virtually surface groups. Notice that $G$ contains a finite-index subgroup $G' \leq G$ which is torsion-free. (For instance, this follows from the fact that $G$ is residually finite \cite[Proposition II.2.12]{MR1954121}.) Such a subgroup then decomposes as a graph of groups such that the edge-groups are trivial and the vertex-groups are either virtually free abelian or virtually surface groups; in other words, $G'$ is a free product of virtually free abelian groups and virtually surface subgroups. Because the factors of such a free product are separable, we conclude that $G'$ (and so $G$) contains a finite-index subgroup which decomposes as a free product of free abelian groups and surface groups.
\end{proof}

\addcontentsline{toc}{section}{References}

\bibliographystyle{alpha}
{\footnotesize\bibliography{CCflatHyp}}

\Address

%

\end{document}